\documentclass{imanum}
\usepackage{graphicx}
\usepackage{algorithm,amsmath,graphicx,epstopdf,graphics,epsf}
\usepackage{caption,float,multicol,bm,enumerate,amsfonts}
\usepackage{algorithm,algpseudocode,pifont,mdframed}
\jno{drnxxx}
\received{29 January 2018}

\begin{document}
\newtheorem{thm}{Theorem}
\newtheorem{defn}{Definition}
\newtheorem{cor}{Corollary}
\newtheorem{eg}{Example}
\newtheorem{lem}{Lemma}
\def\ni{\noindent}
\def\vs{\vspace}
\title{Stability analysis of improved Two-level orthogonal Arnoldi procedure}

\author{%
{\sc
Mashetti Ravibabu\thanks{Corresponding author. Email: mashettiravibabu2@gmail.com,mashettir@iisc.ac.in}
} \\[2pt]
Department of Computational and Data Sciences, Indian Institute of Science,
Bengaluru, India, 560012. 
}
\shortauthorlist{Mashetti Ravibabu}

\maketitle

\begin{abstract}
{ The SOAR method for computing an orthonormal basis of a second-order Krylov subspace can be numerically unstable (see \cite{stabtoar}). 
In the Two-level orthogonal Arnoldi(TOAR) procedure, an alternative to SOAR, the problem of instability had circumvented. 
A stability analysis of the second-order Krylov subspace's orthonormal basis in TOAR with respect to the coefficient matrices of a quadratic problem remain open; see \cite{stabtoar}. This paper proposes the Improved-TOAR method(I-TOAR) and solves the said open problem for I-TOAR.}
%
{Second-order Krylov subspace, Second-order Arnoldi procedure, Backward Stability,
Model order reduction, Dynamical systems.}
\end{abstract}

\section{Introduction}
\label{sec;introduction}

Large-scale quadratic problems are ubiquitous in Linear stability analysis, Model order reduction, Dissipative acoustics, and Constraint least
squares problems; See  \cite{ahuja,femqep,gnad,sima}, 
and \cite{huit}. We suggest the readers refer \cite{tiss} for other applications. 
Two approaches are well-known to solve quadratic eigenvalue problems and quadratic system of equations. One approach finds an appropriate linearization that results in linear eigenvalue problems or a linear system of equations. Another approach  projects larger sparse quadratic problem onto a lower dimensional subspace, and subsequently produce a small, dense QEP or a system of equations. 
The first approach has a drawback that it increases the condition number due to linear problems of double the size; see, e.g., \cite{hwang}. The popular methods such as Residual inverse iteration, Second-Order Arnoldi (SOAR), and Two-level orthogonal Arnoldi (TOAR) methods follow the second  approach; 
See \cite{soar,neum,lock}, and \cite{stabtoar}.

For the given quadratic problem, The SOAR method constructs an orthonormal basis of a second-order Krylov subspace using a recurrence relation an analog to that in the Arnoldi method. It also generates a non-orthonormal basis of a Krylov subspace associated with the corresponding linear problem. To do this, SOAR requires a solution of a triangular linear system, ill-conditioned, in general. Though this causes numerical instability, the SOAR method found applications in Quadratic eigenvalue problems, Structural acoustics analysis, and Model order reduction of second-order
dynamical systems, etc. (see \cite{yang}, \cite{puri}, and \cite{baisu} for further information).

\cite{zhu} and \cite{su} proposed the \emph{Two-Level orthogonal Arnoldi(TOAR) method} to overcome the instability problem in SOAR. As the name suggests, the Gram-Schmidt orthonormalization procedure involved in the two levels of TOAR; in the first and the second levels to construct orthonormal bases for the second-order Krylov subspace and the associated linear Krylov subspace, respectively. 
\cite{stabtoar} proved under some mild assumptions that TOAR with partial reorthogonalization is backward stable to compute an orthonormal basis of associated linear Krylov subspace. However, 
similar stability analysis for the second-order Krylov subspace with respect to the coefficient matrices of a quadratic problem left open; see the Concluding Remarks in \cite{stabtoar}.
%
%

In this paper, we are proposing the Improved TOAR(I-TOAR) method. The I-TOAR method improves TOAR in constructing an orthonormal basis of a second-order Krylov subspace. The proposed improvement is necessary in the TOAR method to solve the said open problem. Using I-TOAR, this paper does the stability analysis for the second-order Krylov subspace with respect to the coefficient matrices of a quadratic problem.



This paper is organized as follows. In Section-2 we briefly discuss the SOAR and TOAR methods and establish relations between the matrix  $Q_k$ and submatrices of $U_k$ those generated by the TOAR method.
Then, Section-3 presents theoretical results those motivated to
improve TOAR.  Section-4 proposes the I-TOAR(Improved
TOAR) algorithm and discusses its implementation details. 
Section-5 does rigorous backward error analysis of I-TOAR in terms of coefficient matrices in quadratic problems for
computing an orthonormal basis of a second-order Krylov
subspace. Section-6 compares the results of numerical experiments
with an application of the TOAR and I-TOAR methods in the Model Order
Reduction of second-order dynamical systems. Section-7 concludes the
paper.

\section{SOAR and TOAR methods}
Let $A$ and $B$ be matrices of order $n,$ and $r_{-1},$ $r_{0}$
be vectors of length $n.$ Then, a sequence of vectors $r_{1},
r_{2}, r_{3},\cdots, $  satisfying the following recurrence relation,
\begin{equation}\label{eq2.1}\relax
r_j = Ar_{j-1}+Br_{j-2}~~\mbox{for}~ j \geq 1
\end{equation}
is called a \textit{second-order Krylov sequence}. The
subspace
\begin{equation}\label{eq2.2}\relax
{G}_k(A,B;r_{-1},r_0) \equiv \mbox{span}\{r_{-1},r_0,r_1,\cdots,
r_{k-1}\}
\end{equation}
is called a \textit{kth second-order Krylov subspace}. A second-order Krylov subspace can be embedded in the linear Krylov subspace $K_k(L,v_0),$ for
\begin{equation}\label{eq2.3}\relax
L=\begin{bmatrix}A~&B\\I~& 0 \end{bmatrix} ~~\mbox{and}~~
v_0=\begin{pmatrix} r_0\\r_{-1} \end{pmatrix},
\end{equation}
where $I$ is an Identity matrix of order $n,$ for details refer to  \cite{stabtoar}.

Let a vector $q_1$ be a linear combination of $r_{-1},$  $r_0,$ and $\|q_1\|=1.$ If a vector $q_{k+1}$ at $k^{th}$ iteration of SOAR is non-zero then it is
orthogonal to the 
set of unit vectors
generated in previous iterations. Further, $\|q_{k+1}\|=1.$ The
non-zero column vectors of 
$Q_{k+1}\equiv[q_1~q_2~\cdots~q_k~q_{k+1}]$ 
form an orthonormal basis for the second-order Krylov subspace
$G_k(A,B;r_{-1},r_0).$
%
In SOAR,    
$P_k:=[p_1,p_2,\cdots p_k]$ is the matrix 
satisfying the following relations:
$$AQ_k+BP_k=Q_kT_k+t_{k+1,k}q_{k+1}e_k^\ast,$$
and
$$Q_k= P_kT_k+t_{k+1,k}p_{k+1}e_k^\ast,$$
where $T_k$ is an upper Hessenberg matrix of
order $k.$  
In compact form these relations can be written as
follows:
$$\begin{bmatrix}A~&B\\I~& 0 \end{bmatrix}\begin{bmatrix}Q_k\\P_k \end{bmatrix}= \begin{bmatrix}Q_k\\P_k \end{bmatrix}H_k+t_{k+1,k} \begin{pmatrix}q_{k+1}\\p_{k+1}\end{pmatrix}e_k^\ast.$$
Observe that the above equation is similar to the Arnoldi decomposition for
the matrix $L$ and the initial vector $v_0.$ However, column vectors
of a matrix $\begin{bmatrix}Q_k\\P_k
\end{bmatrix}$ are non-orthonormal even though they form a basis for the linear Krylov subspace $K_k(L,v_0).$  
The SOAR method requires the solution of ill-conditioned triangular linear system in order to avoid explicit computation of $P_k.$ 
Consequently, this makes the SOAR procedure numerically instable. To circumvent the instability an alternative method proposed in \cite{stabtoar}, the Two-level orthogonalization Arnoldi(TOAR) method. 

The TOAR method  starts with
randomly chosen non-zero initial vector $v_0:=\begin{pmatrix}
r_0\\r_{-1}\end{pmatrix}.$ Then, it finds a rank revealing QR
decomposition of the matrix $[r_{-1}~~r_0]$:
$$[r_{-1}~~r_0]= Q_1 X,$$
where $Q_1$ and $X$ are matrices  of order  $n \times \eta_1$ and
$\eta_1 \times 2,$ respectively. If the vectors $r_{-1}$ and $r_0$ are
linearly independent then $\eta_1=2,$ otherwise $\eta_1=1.$ Following the MATLAB notation define $$U_{1,1}:=X(:,2)/\|v_0\|~\mbox{and}~U_{1,2}:= X(:,1)/\|v_0\|.$$ 
In TOAR the vector $\begin{pmatrix}Q_1 U_{1,1}\\Q_1
U_{1,2} \end{pmatrix}$  forms an orthonormal basis for  $K_1(L,v_0).$ Henceforth, TOAR recursively computes the matrices $Q_k,$ $U_{k,1},$ and $U_{k,2}$ using the relations in the following lemma; the Lemma-3.1 in \cite{stabtoar}.

\ni\begin{lem}\label{lem1}\relax 
Let column vectors of $V_j \equiv \begin{bmatrix}Q_j U_{j,1}\\Q_j
U_{j,2} \end{bmatrix}$ 
form an orthonormal 
basis for  
$K_j(L,v_0),$ for $j=k,k+1.$ Assume that the matrices $V_k$ and $V_{k+1}$ are governed by the following
Arnoldi decomposition of order $k:$
\begin{equation}\label{eq2.4}\relax
LV_k=V_{k+1}\underline{H_{k+1}},
\end{equation}
where $V_{k+1}$ is a matrix consisting $V_k$ in its first $k$ columns, and
$\underline{H_{k+1}}$ is an upper Hessenberg matrix of order $(k+1)
\times k.$  
Then, 
$$\mbox{span}\{Q_{k+1}\}=\mbox{span}\{Q_{k},r\},$$
where $r= AQ_{k}U_{k,1}(:,k)+BQ_{k}U_{k,2}(:,k).$ Furthermore,
\\ (a)
if $r \in \mbox{span}\{Q_{k}\},$ then there exist vectors $x_k$ and
$y_k$ such that
\begin{equation}\label{eq5}\relax
Q_{k+1}=Q_{k},~~~U_{k+1,1}=[U_{k,1}~ x_k],~~~and~~U_{k+1,2}=[U_{k,2}~y_k];
\end{equation}
 (b) otherwise, there exist vectors $x_k,$ $y_k,$ and a scalar $\beta_k \neq 0$ such that
\begin{equation}\label{eq6}\relax
Q_{k+1}=[Q_{k}~~q_{k+1}],~~~U_{k+1,1}=
\begin{bmatrix}U_{k,1}~&x_k\\0~&\beta_k \end{bmatrix}, ~~~and~~~U_{k+1,2}=
\begin{bmatrix}U_{k,2}~&y_k\\0~&0 \end{bmatrix}.
\end{equation}
\end{lem}
Observe from the Lemma-\ref{lem1} that the column vectors of $\begin{bmatrix}
U_{k,1}\\ U_{k,2} \end{bmatrix}$ are orthonormal as $Q_k$ have orthonormal columns, and 
$$\begin{bmatrix}Q_k
U_{k,1}\\Q_k U_{k,2} \end{bmatrix}= \begin{bmatrix}Q_k &~~\\~~&Q_k
\end{bmatrix}\begin{bmatrix}
U_{k,1}\\U_{k,2} \end{bmatrix}.$$
Also observe
from the Lemma-\ref{lem1} that $U_{k,1}$ is an upper Hessenberg matrix if the first column of $U_{k,1}$ has non-zero elements in the rows $1$ and $2.$ 
But, it need not be an
unreduced matrix. On the other hand, $U_{k,1}$ is an upper triangular matrix when the order of a matrix $U_{1,1}$ is $1.$ Similarly, it is easy to see that $U_{k,2}$ is an upper triangular
matrix as $U_{1,2}$ is either a matrix of order $1$ or a matrix of order $2 \times 1$ with a zero element at the bottom. The following lines state this discussion in the form of a lemma.
\begin{lem}\label{lem2}\relax
If the vectors $r_{-1}$ and $r_0$ are linearly independent, then
$U_{k,1}$ is an upper Hessenberg matrix. Otherwise, it is an upper
triangular matrix. $U_{k,2}$ is always an upper triangular matrix.
\end{lem}
Next, by using the equation (\ref{eq2.4}), the following lemma proves that
$U_{k,1}^\ast Q_k^\ast
(AQ_KU_{k,1}+BQ_kU_{k,2})+U_{k,2}^\ast U_{k,1}$ is an upper
Hessenberg matrix. 
\begin{lem}\label{lem3}\relax
Let $Q_k,$ $U_{k,1},$ and $U_{k,2}$ be the same as in the
Lemma-\ref{lem1}. Then, the matrix $U_{k,1}^\ast Q_k^\ast
(AQ_KU_{k,1}+BQ_kU_{k,2})+U_{k,2}^\ast U_{k,1}$ is an upper
Hessenberg matrix.
\end{lem}
\begin{proof}
By considering the equation~(\ref{eq2.4}), we have
\begin{equation}\label{eq7}\relax
\begin{bmatrix}A~&B\\I~&0 \end{bmatrix} \begin{bmatrix}Q_k
U_{k,1}\\Q_k U_{k,2} \end{bmatrix} = \begin{bmatrix}Q_{k+1}
U_{k+1,1}\\Q_{k+1} U_{k+1,2} \end{bmatrix} \underline{H_{k+1}}.
\end{equation}
Now, apply an inner product on both the sides with the matrix
$\begin{bmatrix}Q_k U_{k,1}\\Q_k U_{k,2}
\end{bmatrix}.$ Furthermore, use the fact from the Lemma-
\ref{lem1} that column vectors of $V_j \equiv \begin{bmatrix}Q_j U_{j,1}\\Q_j
U_{j,2} \end{bmatrix}$ are orthonormal for $j=k,k+1$ to get the following:
\begin{equation}\label{2.8}\relax
U_{k,1}^\ast Q_k^\ast (AQ_kU_{k,1}+BQ_kU_{k,2})+U_{k,2}^\ast (Q_k^\ast Q_k)
U_{k,1} =H_k,
\end{equation}
where $H_k$ is a principal submatrix of order $k$ from the top left corner of $\underline{H_{k+1}},$ and is also an upper Hessenberg matrix. Since coulmn vectors of $Q_k$ are orthonormal, the above equation proves the lemma.
Hence the proof is over.
\end{proof}

The following lemma will be helpful later. It derives a relation between the
matrices $U_{k,1},$ $U_{k,2},$ and the vectors $x_k,$ $y_k.$

\begin{lem}\label{lem4}\relax
Let $U_{i,1}$ and $U_{i,2}$ for $i=k,~k+1$ be the same as
in the Lemma-1. Then, $U_{k,1}^\ast x_k= -U_{k,2}^\ast y_k.$
\end{lem}
\begin{proof}
Comparing the last $n$ rows of both the sides of the
equation~(\ref{eq7}) gives
\begin{equation}\label{vig1}\relax
Q_kU_{k,1}= Q_{k+1}U_{k+1,2}\underline{H_{k+1}}.
\end{equation}
Multiply both the sides of the above equation with $Q_k^\ast.$
Since the columns of $Q_k$ are orthonormal, this gives
$$U_{k,1} = [I~~0]U_{k+1,2}\underline{H_{k+1}}.$$
Moreover, by using the structure of $U_{k+1,2}$ from the
equation (\ref{eq6}),  
this implies
$$U_{k,1}= U_{k,2}H_k+h_{k+1,k}y_ke_k^\ast$$
Thus,
\begin{equation}\label{new1}\relax
U_{k,2}^\ast U_{k,1} = U_{k,2}^\ast U_{k,2}H_k+h_{k+1,k}U_{k,2}^\ast
y_ke_k^\ast.
\end{equation}
Consequently, substituting this relation in the equation (\ref{2.8}) gives the following:
$$U_{k,1}^\ast Q_k^\ast (AQ_kU_{k,1}+BQ_kU_{k,2})+ U_{k,2}^\ast U_{k,2}H_k+h_{k+1,k}U_{k,2}^\ast
y_ke_k^\ast = H_k.$$ Further, by using $U_{k,2}^\ast U_{k,2}=
I-U_{k,1}^\ast U_{k,1},$ the above equation becomes as follows:
$$U_{k,1}^\ast Q_k^\ast (AQ_kU_{k,1}+BQ_kU_{k,2})-U_{k,1}^\ast U_{k,1}H_k =
-h_{k+1,k}U_{k,2}^\ast y_k e_k^\ast.$$ Now, compare the first $n$
rows of both sides of the equation~(\ref{eq2.4}) to observe 
$(AQ_kU_{k,1}+BQ_kU_{k,2})= Q_{k+1}U_{k+1,1}\underline{H_{k+1}}.$ By using this, the above equation gives
$$U_{k,1}^\ast Q_k^\ast Q_{k+1}U_{k+1,1}\underline{H_{k+1}}-U_{k,1}^\ast U_{k,1}H_k =
-h_{k+1,k}U_{k,2}^\ast y_k e_k^\ast.$$ As the column vectors of $Q_k$ are
orthonormal, this implies
$$U_{k,1}^\ast [I~~0]U_{k+1,1}\underline{H_{k+1}}-U_{k,1}^\ast U_{k,1}H_k =
-h_{k+1,k}U_{k,2}^\ast y_k e_k^\ast.$$ Because of the structure of
$U_{k+1,1}$ in Lemma-1, the above equation gives
$$U_{k,1}^\ast [U_{k,1}~~x_k]\underline{H_{k+1}}-U_{k,1}^\ast U_{k,1}H_k =
-h_{k+1,k}U_{k,2}^\ast y_k e_k^\ast.$$ Finally, use
$\underline{H_{k+1}}
=\begin{bmatrix}H_k\\h_{k+1,k}e_k^\ast\end{bmatrix} $ to get
$$U_{k,1}^\ast U_{k,1}H_k+h_{k+1,k}U_{k,1}^\ast x_ke_k^\ast -U_{k,1}^\ast U_{k,1}H_k =
-h_{k+1,k}U_{k,2}^\ast y_k e_k^\ast.$$ Therefore,we have proved the
lemma, since $h_{k+1,k} \neq 0.$
\end{proof}
\section{Improved TOAR method}
In this section, we present a few results which are the ground for
proposing the new algorithm, I-TOAR. In I-ToAR, the computation of an
orthonormal basis of a second-order Krylov subspace remains same as
in TOAR, except I-TOAR imposes an additional condition of
orthogonality on the matrix $U_{k,1}.$ In the following
Lemmas, we will prove that imposing such a condition on $U_{k,1}$ prompts 
$U_{k,2}$ to be a diagonal matrix.

\ni\begin{lem}\label{lem6}\relax Let the matrix $Q_kU_{k,1}$ have
orthogonal columns, and the columns of $Q_k U_{k,2}$ are $B-$
orthogonal. If $A$ is symmetric, and $U_{k,1}^\ast U_{k,1}+U_{k,2}^\ast Q_k^\ast BQ_k U_{k,2}$ is an identity matrix, then a
matrix $H_k$ also symmetric.
\end{lem}

\begin{proof}
From the equation (\ref{eq7}) 
we have
$$Q_k^\ast (AQ_kU_{k,1}+BQ_kU_{k,2}) =
U_{k,1}H_k+h_{k+1,k}x_ke_k^\ast.$$ Since the columns of $Q_k$ are
orthonormal, the orthogonality of columns of $Q_kU_{k,1}$ force
the columns of $U_{k,1}$ to be orthogonal. Now, multiply both the
sides of the previous equation from the right with $U_{k,1}^\ast$ to
get
$$U_{k,1}^\ast Q_k^\ast (AQ_kU_{k,1}+BQ_kU_{k,2}) =
U_{k,1}^\ast U_{k,1}H_k.$$ From the equation (\ref{vig1}) note that  $Q_kU_{k,1}= Q_{k+1}U_{k+1,2}\underline{H_{k+1}}.$  On substituting this the above equation gives:
$$U_{k,1}^\ast Q_k^\ast AQ_kU_{k,1}+\underline{H_{k+1}}^\ast U_{k+1,2}^\ast Q_{k+1}^\ast BQ_kU_{k,2} =
U_{k,1}^\ast U_{k,1}H_k.$$ Since the columns of $Q_kU_{k,2}$ are
$B-$ orthogonal, It satisfies the relation:
$$\underline{H_{k+1}}^\ast U_{k+1,2}^\ast Q_{k+1}^\ast BQ_kU_{k,2}
=H_k^\ast U_{k,2}^\ast Q_{k}^\ast BQ_kU_{k,2}.$$ Thus, the two
previous equations together gives the following:
$$U_{k,1}^\ast Q_k^\ast AQ_kU_{k,1}
=U_{k,1}^\ast U_{k,1}H_k-H_k^\ast U_{k,2}^\ast Q_{k}^\ast
BQ_kU_{k,2}.$$ Since $A$ is symmetric, $U_{k,1}^\ast Q_k^\ast
AQ_kU_{k,1}$ also symmetric.
This implies
$$(U_{k,1}^\ast U_{k,1}+U_{k,2}^\ast Q_{k}^\ast
BQ_kU_{k,2})H_k-H_K^\ast (U_{k,1}^\ast U_{k,1}+U_{k,2}^\ast
Q_{k}^\ast BQ_kU_{k,2}) =0.$$ Now, use the fact that $U_{k,1}^\ast
U_{k,1}+U_{k,2}^\ast Q_k^\ast BQ_k U_{k,2}$ is an Identity matrix to
conclude $H_k$ is symmetric. Hence, the proof over.
\end{proof}

The Lemma-\ref{lem6} envisages parallelizing the Symmetric TOAR (STOAR) procedure in \cite{camp}, by parallelly orthogonalizing columns of $U_{k,1}$ and $B-$ orthogonalizing
columns of $Q_k U_{k,2},$  
such that 
$U_{k,1}^\ast U_{k,1}+U_{k,2}^\ast Q_k^\ast BQ_k U_{k,2}$ is an Identity matrix. Though this discussion is interesting, we do not elongate it as latter distracts our attention from the paper. 
The following lemma generalizes the Lemma-\ref{lem6} for non-symmetric QEP.

\begin{lem}\label{lem7}\relax
Follow the notation of the Lemma-\ref{lem1}.
If $U_{k+1,1}$ is an orthogonal
matrix then  
$U_{k,2}$ is a diagonal
matrix.
\end{lem}
\begin{proof}
As $U_{k+1,1}$ is an orthogonal matrix, its principal
submatrix $U_{k,1}$ also orthogonal. Further, orthonormality of the column vectors of $\begin{bmatrix} U_{k,1}\\U_{k,2} \end{bmatrix}$ implies $U_{k,2}^\ast U_{k,2}$ is a diagonal matrix. Recall from the equation (\ref{new1}) that $U_{k,2}^\ast U_{k,1} = U_{k,2}^ \ast U_{k,2}H_k,$ provided 
$U_{k,2}^\ast y_k= 0.$ Observe that $U_{k,2}^\ast y_k= 0$ follows from the Lemma-\ref{lem4}, and the equations (\ref{eq5}), (\ref{eq6}) on using the hypothesis that column vectors of $U_{k+1,1}$ are orthogonal.
%
Now, $U_{k,2}^\ast U_{k,1}$ is an upper Hessenberg matrix as $H_k$ and $U_{k,2}^\ast U_{k,2}$ are upper Hessenberg and diagonal matrices, respectively. 
Further, as $U_{k,1}$ is an upper
Hessenberg matrix from the Lemma-1, this is possible only when $U_{k,2}$ is a diagonal matrix. Therefore, the proof is complete.
\end{proof}

The above lemma is a base to propose I-TOAR  method in the next section for
constructing an orthonormal basis of a linear Krylov subspace
associated with the given quadratic problem.
\section{Implementation}
This section includes two subsections. The first subsection derives relations between entries of the matrices those involved in the two successive iterations of I-TOAR. The second subsection will discuss the Improved TOAR procedure to compute the compact Arnoldi factorization for the given QEP.
\subsection{Matrices in two successive iterations of I-TOAR} Assume that the
initial vectors $r_{-1}~\mbox{and}~ r_0$ are chosen randomly such that
$v_0:=\begin{bmatrix}r_0\\r_{-1}\end{bmatrix} \neq 0.$ 
The I-TOAR method computes a $QR$ decomposition of the $n
\times 2$ matrix $[r_{-1}~~r_0]:$
$$[r_{-1}~~r_0]=Q_1X,$$
where $Q_1$ is an orthonormal matrix of order $n \times \alpha$ and
$X$ is a matrix of size $\alpha \times 2.$ Here, $\alpha = 2,$ when the
vectors $r_{-1}$ and $r_0$ are linearly independent, otherwise
$\alpha=1.$ Now, define a matrix $V_1$ as follows:
$$V_1 = \frac{1}{\gamma}
\begin{bmatrix}r_0\\r_{-1}\end{bmatrix}= \frac{1}{\gamma} \begin{bmatrix}Q_1X( :
, 2) \\Q_1X( : , 1)\end{bmatrix} = \begin{bmatrix} Q_1&\\ & Q_1
\end{bmatrix} \begin{bmatrix} U_{1,1}\\ U_{1,2} \end{bmatrix} \equiv Q_{[1]}U_1,$$ where $\gamma = \|v_0\|_2,$ $U_{1,1} = X( : ,
2)/\gamma,$ and $U_{1,2} = X( : , 1)/\gamma$. After the
initialization, we have $Q_1, U_1$, and an empty $1 \times 0$
Hessenberg matrix $H_1 = [ ~~].$

\ni Assume that we have the compact Arnoldi decomposition
for $k=j,~j\geq 1.$ Similar to the TOAR method, I-TOAR enlarges this
decomposition for $k=j+1$ by 
computing an orthogonal matrix
$Q_{j+1}=[Q_j~~q_{j+1}]$ such that
\begin{equation}\label{new3}\relax
span\{Q_j,q_{j+1}\}=span\{Q_j,r\},~\mbox{where}~r=
AQ_jU_{j,1}(:,j)+BQ_jU_{j,2}(:,j).
\end{equation}
Thus, a vector $q_{j+1}$ is given by:
\begin{equation}\label{eq8}\relax
q_{j+1}= (r-Q_js)/\beta~~\mbox{ with}~ s = Q_j^*r,~ \beta = \|r -
Q_js\|_2.
\end{equation}
Here, it is assumed that $\beta \neq 0.$ Note that, if
$\beta =0,$ then deflation occurs and $r \in span\{Q_j\}.$ We state
this discussion  in the form of the following lemma.

\begin{lem}\label{lem8}\relax
Let for $i=j,j+1,$ column vectors of a matrix $Q_i$ form an orthonormal basis of a second-order Krylov subspace at the $i^{th}$ iteration of I-TOAR. Then, In case of no deflation, $Q_{j+1}=[Q_j~~q_{j+1}],$ where $q_{j+1}$ is given by the equation (\ref{eq8}). Otherwise $Q_{j+1}=Q_j.$
\end{lem}

The following lemma establishes relations between the matrices $U_i,~i=j,j+1$ in the I-TOAR method.

\begin{lem}\label{lem9}\relax
Let $U_{i,1},~U_{i,2},~i=j,j+1$ be matrices at the
iterations $j,j+1$ of I-TOAR and of the form  described in the
equation (\ref{eq6}). Then, entries of an upper Hessenberg matrix
$\underline{H_{j+1}}$ in I-TOAR satisfies the following relations:
\begin{equation}\label{eq9}\relax
 h_j=(U_{j,1}^\ast U_{j,1})^{-1}U_{j,1}^\ast s= (U_{j,2}^\ast U_{j,2})^{-1}U_{j,2}^\ast u=U_{j,1}^\ast s+U_{j,2}^\ast u,
\end{equation}
and
\begin{equation}\label{eq10}\relax
h_{j+1,j}^2=\|s-U_{j,1}h_j\|^2+\|u-U_{j,2}h_j\|^2+\beta^2.
\end{equation}
where $u= U_{j,1}(:,j).$
\end{lem}
\begin{proof}
\ni As column vectors of the matrix $Q_{k+1}$ are orthonormal, observe the following relation from the equations (\ref{eq7}) and (\ref{new3}):
\begin{equation}\label{eq11}\relax
\begin{bmatrix}Q_{j+1}^\ast r\\Q_{j+1}^\ast Q_j U_{j,1}(;,j)
\end{bmatrix}= \begin{bmatrix} U_{j+1,1}\\U_{j+1,2}
\end{bmatrix}\underline{H_{j+1}}(:,j).
\end{equation}
The I-TOAR method uses this relation 
to compute the matrices $U_{j+1,1}$ and $U_{j+1,2},$ 
which are of the following form:
$$U_{j+1,1}=\begin{bmatrix} U_{j,1}&x_{j}\\0&\beta_{j}
\end{bmatrix} ~\mbox{and}~U_{j+1,2}=\begin{bmatrix}U_{j,2}&y_{j}\\0&0\end{bmatrix}.$$
From the previous two equations, it is clear that matrix $U_{j,1},$
$x_{j},$ and $\beta_{j}$ satisfy the following equation:
\begin{equation}\label{eq11a}\relax
Q_{j+1}^\ast r := \begin{bmatrix}s \\ \beta \end{bmatrix} =
\begin{bmatrix} U_{j,1}h_j+h_{j+1,j}x_{j}\\h_{j+1,j} \beta_{j}
\end{bmatrix},
\end{equation}
where $h_j={H}_{j+1}(1:j,j)$ and $h_{j+1,j}={H}_{j+1}(j+1,j).$
Since the matrix $U_{j+1,1}$ is orthogonal in I-TOAR, the two previous equations together implies $U_{j,1}^\ast x_{j}=0,$ and
\begin{equation}\label{eq12}\relax
U_{j,1}^\ast s = U_{j,1}^\ast U_{j,1}h_j.
\end{equation}
Similarly, comparing the last
row of the equation  (\ref{eq11}) gives the following using an orthonormal property of the matrix $Q_{j+1},$ and the structure of $U_{j+1,2}:$
\begin{equation}\label{eq12a}\relax
\begin{bmatrix} u\\0\end{bmatrix}=\begin{bmatrix}
U_{j,2}h_j+h_{j+1,j}y_{j}\\0\end{bmatrix},
\end{equation}
 where $u\equiv
U_{j,1}(:,j).$ By using the Lemma-\ref{lem4}  and $U_{j,1}^\ast x_{j}=0,$ 
this gives
\begin{equation}\label{eq13}\relax
U_{j,2}^\ast u = U_{j,2}^\ast U_{j,2}h_j.
\end{equation}
Now adding the equations  (\ref{eq12}) and (\ref{eq13}) based on the fact that column vectors of $\begin{bmatrix}U_{j,1}\\U_{j,2}\end{bmatrix}$ are orthonormal gives the following relation:
\begin{equation}\label{eq13b}\relax
h_j= U_{j,1}^\ast s+U_{j,2}^\ast u.
\end{equation}
Similarly, recall the following from the equations
(\ref{eq11a}) and (\ref{eq12a}):
\begin{equation}\label{eq13a}\relax
s-U_{j,1}h_j = h_{j+1,j}x_{j+1}~~ \mbox{and}~~ u-U_{j,2}h_j=
h_{j+1,j}y_{j+1}.
\end{equation}
By using the fact that $(x_{j+1}~\beta_{j+1}~y_{j+1}~0)'$ is a column vector of an orthonormal matrix $\begin{bmatrix}U_{j+1,1}\\U_{j+1,2} \end{bmatrix},$ It gives the following: 
$$ \|s-U_{j,1}h_j\|^2+\|u-U_{j,2}h_j\|^2+h_{j+1,j}^2 \beta_{j+1}^2=h_{j+1,j}^2.$$
Hence, the equation (\ref{eq10}) is proved by observing from the equation
(\ref{eq11a}) that $h_{j+1,j}\beta_{j+1} =\beta.$ 
Similarly, observe that the equation (\ref{eq9}) follows from the equations
(\ref{eq12}), (\ref{eq13}), and (\ref{eq13b}). Therefore, the proof is complete.
\end{proof}

The Lemma-\ref{lem9} gives the relations to transit from $j^{th}$ to $(j+1)^{th}$ iteration in I-TOAR, provided there is no deflation  at the $(j+1)^{th}$ iteration. In the following, we derive similar expressions for computing the matrix $\underline{H_{j+1}}$ in
I-TOAR, in case of deflation.

\begin{lem}\label{lem10}\relax
Let $\mbox{for}~i=j,j+1,$ $U_{i,1}$ and $U_{i,2}$ be matrices same as in the
Lemma-\ref{lem9}, but of the form described in the equation (\ref{eq5}). Then,
entries of an upper Hessenberg matrix $\underline{H_{j+1}}$ in I-TOAR satisfies the equation (\ref{eq9}), and also the following one:
\begin{equation}\label{new4}\relax
h_{j+1,j}^2 = \|s-U_{j,1}h_j\|^2+\|u-U_{j,2}h_j\|^2 =  \|(I-U_{j,1}U_{j,1}^\dag) s\|^2+\|(I-U_{j,2}U_{j,2}^\dag)u\|^2.
\end{equation}
\end{lem}
\begin{proof}
Recall from the Lemma-\ref{lem1} that in the case of deflation, $Q_{j+1}=Q_j,$ and $U_{j+1,1,}$ $U_{j+1,2}$ are of the following form:
$$U_{j+1,1}= [ U_{j,1}~~x_{j} ]~~\mbox{and}~~U_{j+1,2} = [U_{j,2}~~y_{j}].$$
These equations together with the equation (\ref{eq11}) gives the following relations:
\begin{equation}\label{eq14a}\relax
Q_{j+1}^\ast r = Q_j^\ast r = s=U_{j,1}h_j+h_{j+1,j}x_{j},
\end{equation}
and
\begin{equation}\label{eq15a}\relax
 U_{j,1}(:,j)=u=U_{j,2}h_j+h_{j+1,j}y_{j}.
\end{equation}
These relations are similar to the equations (\ref{eq11a}) and (\ref{eq12a}) in the previous lemma. As $\begin{pmatrix} x_{j+1}\\y_{j+1} \end{pmatrix}$ is a column vector of an orthonormal matrix $\begin{bmatrix}U_{j+1,1}\\U_{j+1,2} \end{bmatrix},$ we have $\|x_{j+1}\|^2+\|y_{j+1}\|^2=1.$ Using this, the previous two equations gives:
$h_{j+1,j}^2 = \|s-U_{j,1}h_j\|^2+\|u-U_{j,2}h_j\|^2.$ Now, see the following relations to observe this is equal to $\|(I-U_{j,1}U_{j,1}^\dag) s\|^2+\|(I-U_{j,2}U_{j,2}^\dag)u\|^2.$
\begin{equation}\label{eq14}\relax
U_{j,1}^\ast s = U_{j,1}^\ast U_{j,1}h_j~~\mbox{and}~~U_{j,2}^\ast u = U_{j,2}^\ast U_{j,2}h_j.
\end{equation}
The above equation follows from the equations (\ref{eq14a}) and (\ref{eq15a}) by using the identities $U_{j,1}^\ast x_j=0=U_{j,2}^\ast y_j.$ These identities follows from the fact that the matrix $U_{j+1,1}$ is orthogonal, and the Lemma-\ref{lem4}. Therefore, we proved (\ref{new4}). Now, it is required to prove the equation (\ref{eq9}).


\ni As the matrix $\begin{bmatrix}U_{j,1}\\U_{j,2}
\end{bmatrix}$ has orthonormal columns, $U_{j,1}^\ast U_{j,1}+U_{j,2}^\ast U_{j,2}$ is an Identity matrix. Using this, the equation
(\ref{eq14}) gives
$U_{j,1}^\ast s+U_{j,2}^\ast u= h_j.$ In turn, this equation together with (\ref{eq14}) proves the equation (\ref{eq9}). Hence, the proof is over.
\end{proof}

 The equation (\ref{new4}) in the Lemma-\ref{lem10} shows that $h_{j+1,j}^2=0,$ that means, the I-TOAR algorithm break down when the vectors $s=Q_j^\ast r$ and $u$ are in the range space of
matrices $U_{j,1}$ and $U_{j,2},$ respectively. Otherwise, since the column vectors of $Q_j$ are orthonormal, the equation (\ref{eq14}) shows that the vector $h_j$ is the least squares approximation to the vectors $s$ and $u$ from the range space of $Q_jU_{j,1}$ and $Q_jU_{j,2},$ respectively.

\subsection{I-TOAR implementation}
This subsection uses the results of the subsection-4.1 to discuss computational details of I-TOAR. Then, it proposes the I-TOAR algorithm.
Though I-TOAR follows the TOAR for
constructing an orthonormal basis of the second-order Krylov
subspace, 
I-TOAR can build an
orthonormal basis for the associated linear Krylov subspace in various
ways. For example, the following is the one which requires
the least computation compared to all other procedures.

\ni \textbf{Procedure-1:} Use the second equality relation of the equation (\ref{eq9}) to compute the vector $h_j.$  It requires at most $3(j+1)$ flops. Now, use first relation in the equation
(\ref{eq13a}) to compute the vector $h_{j+1,j}x_{j+1},$ and then comparing only the last element in the second relation of the equation (\ref{eq13a}) gives the vector $h_{j+1,j}y_{j+1}.$  Note that, $y_{j+1}$ is a column vector of a diagonal matrix $U_{j+1,2}$ and has only one non-zero entry. This approach takes at most $(j+1)^2+2(j+1)$ flops to compute $h_{j+1,j}x_{j+1},$ and $h_{j+1,j}y_{j+1}.$ After that, as already $\beta$ is known from the equation (\ref{eq11a}), it requires
another $2(j+1)$ flops to compute $h_{j+1,j}.$ Then, scaling of the vectors $h_{j+1,j}x_{j+1},$ and $h_{j+1,j}y_{j+1}$ require $2(j+1)$ flops. Thus, overall computational cost is at most $(j+1)^2+9(j+1)$ flops for this procedure.

However, when diagonal elements of $U_{j,2}$ are smaller in
magnitude, the division operation in the equation
(\ref{eq13}) introduce large floating point arithmetic errors into this procedure.

%
\ni \textbf{Procedure-2:} Instead of the second relation as in the Procedure-1, use the first relation of the equation (\ref{eq9}) for computing $h_j.$ Then, follow the Procedure-1 to compute $x_{j+1}, y_{j+1},$ and $h_{j+1,j}.$ The computational cost of this procedure is equal to that of the Procedure-1. 

Observe that the Procedures- 1 and 2 requires the
less computation compared to $6(j+1)^2+3(k+1)$ flops required by TOAR to complete the same task(See, \cite{stabtoar}). However, the computed column vectors of  $U_{j,1}$ are may not orthogonal, as both the procedures do not  orthogonalize these vectors explicitly.
The following algorithm uses the Modified Gram-Schmidt(MGS) process to do  explicit orthogonalization. In the
steps (k)-(n), it orthogonalizes the vector $s=Q_j^\ast r$ against
all the columns of $U_{j,1}$ to compute the vector $x_{j+1}.$ Though MGS increase the computation cost, the overall computation cost in the following algorithm to perform the task that of the Procedure-1 is $5(j+1)^2+4(j+1)$ flops. Still, it is less compared to the TOAR method.

%
%

\begin{algorithm}[!htb]
\caption{I-TOAR} \label{Improved TOAR method}
\begin{enumerate}
\item [1.] {\em Start:}Matrices $A$, $B$ and initial length n-vectors
$r_{-1}$ and $r_0$ with $\begin{pmatrix}r_{-1}\\r_0 \end{pmatrix}
\neq 0.$
\item [2.] {\em Output:} $Q_k \in \mathbb{R}^{n \times \alpha_k}$,
$U_{k,1},U_{k,2} \in \mathbb{R}^{\alpha_k \times n},$ and
$\underline{H_k}=\{h_{i,j}\}\in \mathbb{R}^{k \times k-1}.$
\begin{enumerate}[(a)]
\item rank revealing QR: $[r_{-1} ~r_0]=QX$ with $\alpha_1$ being the
rank.
\item Initialize $Q_1=Q,$ $U_{1,1}=X(:,2)/\gamma$ and $U_{1,2}=X(:,1)/\gamma.$
\item for j=1,2,\ldots,k-1 do
\item $r=A(Q_jU_{j,1}(:,j))+B(Q_jU_{j,2}(:,j))$
\item for i= 1,\ldots,$\alpha_j$ do
\item $~~~s_i=q_i^\ast r$
\item $~~~r=r-s_iq_i$
\item end for
\item $\beta = \|r\|_2$
\item Set $s= x:=[s_1 \cdots s_{\alpha_j}]^T$ and $u=U_{j,1}(:,j)$
\item for t=1,2, \ldots,i
\item $ \gamma_1(t) = U_{j,1}(:,t)^*s/\|U_{j,1}(:,t)\|^2;$
\item $s= s-\gamma_1(t) s$
\item end for
\item $h= U_{j,1}^\ast (s-x)+U_{j,2}^\ast U_{j,1}(:,j);~~u(\alpha_j)= u(\alpha_j)-e_{\alpha_j}^\ast U_{j,2}h;$
\item for t=1, 2, \ldots, $\alpha_{j}-1$
\item u(t)=0
\item end for
\item $h_{j+1,j}=(\beta^2+\|s\|_2^2+\|u\|_2^2)^{1/2}$
\item if $h_{j+1,j}=0$ then stop(breakdown) end if
\item if $\beta =0$ then
 $\alpha_{j+1}=\alpha_j $ (deflation)
\item $Q_{j+1}=Q_j;$ $U_{j+1,1} = [U_{j,1}~~s/h_{j+1,j}];$
$U_{j+1,2} =[U_{j,2}~~u/h_{j+1,j}]$
\item else
 $\alpha_{j+1}=\alpha_j+1$
\item $Q_{j+1}=[Q_j~\frac{r}{ \beta}];\\
U_{j+1,1}=\begin{bmatrix} U_{j,1}&s/h_{j+1,j}\\0&\beta/h_{j+1,j}
\end{bmatrix};$ $U_{j+1,2}=\begin{bmatrix} U_{j,2}&u/h_{j+1,j}\\0&0
\end{bmatrix};$
\item end if
\item end for
\end{enumerate}
\end{enumerate}
\end{algorithm}

Similar to the TOAR algorithm in \cite{stabtoar}, the Algorithm-\ref{Improved
TOAR method} also uses the MGS process to maintain orthonormality of the matrix $Q_j.$ The computed $Q_j$ may not orthonormal up to machine precision. To keep the level of orthonormality of $Q_j$ as close to machine precision as possible it is required to reorthogonalize a vector $r$ in the step-(g) against the columns of $Q_j$ by inserting the following code segments between the steps (e) and (h) of the Algorithm-\ref{Improved TOAR method}. 
\begin{center}
for i = 0, . . . , $\eta_j$ \\
$\tilde{s}_i = q^\ast_i r~~~~$\\
$r = r-\tilde{s}_i q_i$\\ $s_i=\tilde{s}_i + s_i ~$\\
end for~~~~~~~~~~~~~~~~\\
$\alpha=\|r\|_2~.$
\end{center}
To end this section, note that it would be possible to apply a similar reorthogonalization procedure between the steps (k) and (n) of the Algorithm-1, to keep the orthogonality of columns of the matrix $U_{j,1},$ close to machine precision. In the next section, we provide a rigorous backward error analysis of the
Algorithm-\ref{Improved TOAR method}.

\section{Backward Error Analysis}
This section provides backward error analysis of the I-TOAR
algorithm, in the presence of finite precision arithmetic. 
The backward error analysis for the associated linear Krylov
subspace in TOAR presented in \cite{stabtoar}, is also valid for I-TOAR with insignificant changes. So, this section provides backward error
analysis only for the second-order Krylov subspace in I-TOAR, in terms of the matrix pair $(A,B).$ That means, this section proves that the
computed basis $\hat{Q}_j$ in I-TOAR is an exact basis matrix of
${G}_k(A+\triangle A,B+\triangle B;r_{-1},r_0)$ with small
$\|[\triangle A~\triangle B]\|_2.$  Note that, for TOAR this is an open problem; See \cite{stabtoar}.

Let us assume that by taking the floating point errors into account, the  Compact Arnoldi decomposition computed by I-TOAR satisfies the following:
$$\begin{bmatrix}A&B\\I&0\end{bmatrix}\begin{bmatrix}\bm\hat{Q}_{k-1}\bm\hat{U}_{k-1,1}\\\bm\hat{Q}_{k-1}\bm\hat{U}_{k-1,2}\end{bmatrix}=\begin{bmatrix}\bm\hat{Q}_{k}\bm\hat{U}_{k,1}\\\bm\hat{Q}_{k}\bm\hat{U}_{k,2}\end{bmatrix}\bm\hat{\underline{H_k}}+E,$$
where matrices with $~~\bm\hat{}~~$ on the top are the computed
matrices counterpart to the matrices in the exact arithmetic, and $E$
is the error matrix. Now,
we introduce two
matrices $F_{mv}$ and $F$ that represent floating point error of matrix-vector product and orthogonalization process, respectively.
\begin{equation}\label{eq16}\relax
F_{mv}=A(\bm\hat{Q}_{k-1}\bm\hat{U}_{k-1,1})+B(\bm\hat{Q}_{k-1}\bm\hat{U}_{k-1,2})-\bm\hat{R}_{k-1},
\end{equation}
\begin{equation}\label{eq17}\relax
F:= \begin{bmatrix} F_1\\F_2\end{bmatrix} = \begin{bmatrix}
\bm\hat{R}_{k-1}
\\\bm\hat{Q}_{k-1}\bm\hat{U}_{k-1,1}\end{bmatrix}-\begin{bmatrix}\bm\hat{Q}_{k}\bm\hat{U}_{k,1}\\\bm\hat{Q}_{k}\bm\hat{U}_{k,2}\end{bmatrix}\bm\hat{\underline{H_k}},
\end{equation}
where $\bm\hat{R}_{k-1}=[\hat{r}_1,\hat{r}_2,\cdots,\hat{r}_{k-1}],$
and $\hat{r}_j =
fl(A\bm\hat{Q}_{j}\bm\hat{U}_{j,1}(:,j))+B\bm\hat{Q}_{j}\bm\hat{U}_{j,2}(:,j)).$

Note that the matrices $\triangle A_1:= -F_{mv}\alpha (\bm\hat{Q}_{k-1}\bm\hat{U}_{k-1,1})^\dag$ and $\triangle B_1:=-F_{mv} (1-\alpha)(\bm\hat{Q}_{k-1}\bm\hat{U}_{k-1,2})^\dag$ satisfy the relation
$$(A+\triangle A_1)(\bm\hat{Q}_{k-1}\bm\hat{U}_{k-1,1})+(B+\triangle B_1)(\bm\hat{Q}_{k-1}\bm\hat{U}_{k-1,2})=\bm\hat{R}_{k-1},$$
where $\alpha$ is some non-zero scalar.
Then, using the equations (\ref{eq16}) and (\ref{eq17}) the matrices $\triangle A_1$ and $\triangle B_1$  satisfy the following relations, respectively:
$$\triangle A_1(\bm\hat{Q}_{k-1}\bm\hat{U}_{k-1,1})+\triangle B_1
(\bm\hat{Q}_{k-1}\bm\hat{U}_{k-1,2}) = -F_{mv},$$
and 
$$(A+\triangle A_1)(\bm\hat{Q}_{k-1}\bm\hat{U}_{k-1,1})+(B+\triangle B_1)(\bm\hat{Q}_{k-1}\bm\hat{U}_{k-1,2})-\bm\hat{Q}_{k}\bm\hat{U}_{k,1}\hat{\underline{H_k}}=F_1.$$
Thus, from these two equations, it is easy to see that
$F_1+F_{mv}$ is the overall error matrix in the orthogonalization
process. Furthermore, on the introduction of two matrices $\triangle
A_2:= F_{1}\alpha (\bm\hat{Q}_{k-1}\bm\hat{U}_{k-1,1})^\dag$ and
$\triangle B_2:=
F_{1}(1-\alpha)(\bm\hat{Q}_{k-1}\bm\hat{U}_{k-1,2})^\dag,$ the previous
equation becomes as follows:
\begin{equation}\label{eq17c}\relax
(A+\triangle A_1+\triangle
A_2)(\bm\hat{Q}_{k-1}\bm\hat{U}_{k-1,1})+(B+\triangle B_1+\triangle
B_2)(\bm\hat{Q}_{k-1}\bm\hat{U}_{k-1,2})=\bm\hat{Q}_{k}\bm\hat{U}_{k,1}\hat{\underline{H_k}}.
\end{equation}
Now, define $[\triangle A~~\triangle
B]:=[\triangle A_1+\triangle A_2~~\triangle B_1+\triangle B_2].$
Then, the equation (\ref{eq17c}) shows that the basis matrix $\bm\hat{Q}_k$ computed in the I-TOAR algorithm is an exact basis matrix of a second-order Krylov subspace ${G}_k(A+\triangle A,B+\triangle B;r_{-1},r_0).$ 
Next, to prove the I-TOAR algorithm is backward stable, it is
required that the relative backward error
$\frac{\|[\triangle A~~\triangle B]\|}{\|[A~~B]\|}$ is of the order
of the machine precision $\varepsilon.$ To verify this, we need to derive
upper bounds for $\|F_{mv}\|_F$ and $\|F_1\|_F.$ 

We adopt the following standard  model for rounding errors in the floating point arithmetic. Let $\alpha$ and $\beta$ be any two real scalars. Then,
$$fl(\alpha~~op~~\beta) =(\alpha ~~op~~\beta)(1+\delta)~~\mbox{with}~~|\delta| \leq \varepsilon~~\mbox{for}~~op=+,-,*,/,  $$
where $fl(x)$ denotes the computed quantity, and $\varepsilon$
denotes the machine precision. We will use the following lemma also in the backward error analysis.  
\begin{lem}\label{lem4.1}\relax
(a). For $x,y \in \mathcal{R}^n,~fl(x+y)~=~x+y+f,$ where $\|f\|_2
\leq (\|x\|_2+\|y\|_2).$\\
(b). For $X \in \mathcal{R}^{n \times k}$ and $y \in \mathcal{R}^k$,
$fl(Xy)~=~Xy+w,$ where $\|w\|_2
\leq k\|X\|_F\|y\|_2 \varepsilon+\mathcal{O}(\varepsilon^2).$\\
(c). For $X \in \mathcal{R}^{n \times k},$ $y \in \mathcal{R}^k,~b
\in \mathcal{R}^n,$ and $\beta \in \mathcal{R},$ $\hat{c} \equiv
fl((b-Xy)/\beta)$ satisfies
$$\beta \hat{c} = b-Xy+g, ~~ \|g\|_2 \leq (k+1)\|[X~~\hat{c}]\|_F \Big\|\begin{bmatrix}y \\ \beta \end{bmatrix}\Big\|_2 \varepsilon+\mathcal{O}(\varepsilon^2). $$
\end{lem}

For the Proof of the Lemma-\ref{lem4.1}, see the Lemma-4.1 in \cite{stabtoar}. This lemma holds true for I-TOAR as well. Next, the following lemma derives an upper bound for $\|F_{mv}\|_F,$
where $F_{mv}$ is a matrix defined as in the equation (\ref{eq16}).

\begin{lem}\label{lem11}\relax
Let $\bm\hat{Q}_{k-1}$ and $\bm\hat{U}_{k-1}$ be orthonormal matrices computed by the I-TOAR procedure. Then,
$$\|F_{mv}\|_F \leq  4k^2n
\|[A~~B]\|_F\|\bm\hat{Q}_{k-1}\|_2   \|\bm\hat{U}_{k-1}\|_2
\varepsilon + \mathcal{O}(\varepsilon^2).$$
\end{lem}
\begin{proof}
From  the definition of $F_{mv}$ in the equation (\ref{eq16}), we have
$$F_{mv}(;,j)= [A~~B]\begin{bmatrix}\bm\hat{Q}_{j}\bm\hat{U}_{j,1}(:,j)\\ \hat{Q}_{j}\bm\hat{U}_{j,2}(:,j)
\end{bmatrix}-\bm \hat{r}_j,$$
where $\bm\hat{r}_j \equiv fl \Big([A~~B]\begin{bmatrix}\bm\hat{Q}_{j}\bm\hat{U}_{j,1}(:,j)\\
\hat{Q}_{j}\bm\hat{U}_{j,2}(:,j)\end{bmatrix}\Big).$ Here, we used the
fact that $\bm\hat{U}_{k-1,1},$ and $\bm\hat{U}_{k-1,2}$ are upper Hessenberg and diagonal matrices, respectively and $j \leq k-1$. Now, by the repeated application of the Lemma-\ref{lem4.1}, we have
$$\bm\hat{r}_j = [A~~B]fl \Big(\begin{bmatrix}\bm\hat{Q}_{j}\bm\hat{U}_{j,1}(:,j)\\ \hat{Q}_{j}\bm\hat{U}_{j,2}(:,j) \end{bmatrix}\Big)+w_j$$
$$~~~~~~~~~~~~~~~~~~=[A~~B]\begin{bmatrix}\bm\hat{Q}_{j}\bm\hat{U}_{j,1}(:,j)\\ \hat{Q}_{j}\bm\hat{U}_{j,2}(:,j)
\end{bmatrix}+w_j^{(1)}+w_j^{(2)}+w_j,$$
where $w_j^{(1)},$ $w_j^{(2)},$ and  $w_j$ are the floating point
error vectors satisfying the following relations, respectively.
$$\|w_j^{(1)} \|_2 \leq k \| \bm\hat{Q}_{j}\|_F\|\bm\hat{U}_{j,1}(:,j)\|_2 \varepsilon +  \mathcal{O}(\varepsilon^2), $$
$$\|w_j^{(2)} \|_2 \leq k \| \bm\hat{Q}_{j}\|_F\|\bm\hat{U}_{j,2}(:,j)\|_2 \varepsilon +  \mathcal{O}(\varepsilon^2), $$
and
$$\|w_j\|_2 \leq 2nk \|[A~~B]\|_F  \Big\|\begin{bmatrix}\bm\hat{Q}_{j}\bm\hat{U}_{j,1}(:,j)\\ \hat{Q}_{j}\bm\hat{U}_{j,2}(:,j) \end{bmatrix} \Big\|_2  \varepsilon +  \mathcal{O}(\varepsilon^2) \leq 4nk \|[A~~B]\|_F \|\bm\hat{Q}_{j}\|_F \Big\|\begin{bmatrix}\bm\hat{U}_{j,1}(:,j)\\ \bm\hat{U}_{j,2}(:,j) \end{bmatrix} \Big\|_2 \varepsilon +  \mathcal{O}(\varepsilon^2).$$
Now, combine all the three previous inequalities, and use the facts $\bm\hat{U}_{j} = \begin{bmatrix}
\bm\hat{U}_{j,1}\\\bm\hat{U}_{j,2}
\end{bmatrix},$ and $\|\bm\hat{U}_{j,i}(:,j)\|_2 \leq \|\bm\hat{U}_{j}(:,j)\|_2,~\mbox{for}~i=1,2.$ It gives
$$\|F_{mv}(:,j)\|_2 \leq (4nk \|[A~~B]\|_F +2k)\|\bm\hat{Q}_{j}\|_F \Big \|\begin{bmatrix}\bm\hat{U}_{j,1}(:,j)\\ \bm\hat{U}_{j,2}(:,j) \end{bmatrix} \Big\|_2 \varepsilon +  \mathcal{O}(\varepsilon^2).$$
Then, using $\|\bm\hat{Q}_{j}\|_F \leq
\|\bm\hat{Q}_{k-1}\|_F$ for $j \leq k-1,$ the above inequality gives the following:
$$\|F_{mv}(:,j)\|_2 \leq  2k(2n \|[A~~B]\|_F +1)\|\bm\hat{Q}_{k-1}\|_F   \|\bm\hat{U}_j(:,j )\|_2 \varepsilon +  \mathcal{O}(\varepsilon^2).$$
Observe that in matrix terms this equation can be written as follows:
$$\|F_{mv}\|_F = \Big(\displaystyle \sum_{j=1}^{k-1}
\|F_{mv}(:,j)\|_2^2 \Big)^{1/2} \leq 2k(2n
\|[A~~B]\|_F+1)\|\bm\hat{Q}_{k-1}\|_F   \|\bm\hat{U}_{k-1}\|_F
\varepsilon + \mathcal{O}(\varepsilon^2).$$ 
Further, using the inequalities $\|\bm\hat{Q}_{j}\|_F \leq \|\bm\hat{Q}_{k-1}\|_F \leq
\sqrt k \|\bm\hat{Q}_{k-1}\|_2,$ and $\|\bm\hat{U}_{k-1}\|_F \leq
\sqrt k \|\bm\hat{U}_{k-1}\|_2$ this gives
$$\|F_{mv}\|_F \leq 2k^2(2n
\|[A~~B]\|_F+1)\|\bm\hat{Q}_{k-1}\|_2   \|\bm\hat{U}_{k-1}\|_2
\varepsilon + \mathcal{O}(\varepsilon^2).$$ Since the I-TOAR Algorithm uses the MGS process to
generate orthonormal matrices $\bm\hat{Q}_{k-1}$ and $\bm\hat{U}_{k-1}$,
we have $\|\bm\hat{Q}_{k-1}\|_2=
\|\bm\hat{U}_{k-1}\|_2=1+\mathcal{O}(\varepsilon).$ Therefore,
neglecting the term $2k^2\|\bm\hat{Q}_{k-1}\|_2
\|\bm\hat{U}_{k-1}\|_2 \varepsilon$ in the above equation completes the proof of the lemma.
\end{proof}
\\\\\ni Next, to derive an upper bound for $\|F_1\|_F,$ the following lemma is required.
\begin{lem}\label{lem12}\relax
Let $f$ be an error vector resulting from the computation in the
steps (k)-(n) of the Algorithm-\ref{Improved TOAR method}. Then
$$\|\bm\hat{U}_{j,1}\|_F\|f\|_2 \leq 2(j+1)^3
\mathcal{K}_2(\bm\hat{U}_{j,1})^2 \varepsilon.$$
\end{lem}
\begin{proof}
The elements of computed vector $\alpha_1$ in the step-(l) of the Algorithm-\ref{Improved TOAR method} satisfy the relation:
$$\gamma_1(i)= \bm\hat{U}_{j,1}(:,i)^\ast
s/\bm\hat{U}_{j,1}^\ast \bm\hat{U}_{j,1}(i,i)+f_i,$$  where  $f_i$
is the error resulting from the computation of an inner product in
the numerator and the
division. 
By using the fact that $\bm\hat{U}_{j,1}^\ast \bm\hat{U}_{j,1}$ is a diagonal matrix, the above element-wise computation can be written in the vector form as follows:
$$\gamma_1 = (\bm\hat{U}_{j,1}^\ast \bm\hat{U}_{j,1})^{-1}  \bm\hat{U}_{j,1}^\ast s+f,$$ 
where an error vector $f$ satisfies the following relation:
$$\|f\|_2 \leq 2(j+1) \|\bm\hat{U}_{j,1}^\ast\|_F\|\bm\hat{s}\|\|_2(\bm\hat{U}_{j,1}^\ast \bm\hat{U}_{j,1})^{-1}\|_F\varepsilon.$$
Therefore, using $ \|\bm\hat{U}_{j,1}\|_F =
\|\bm\hat{U}_{j,1}^\ast\|_F,$ this gives the inequality
$$ \|\bm\hat{U}_{j,1}\|_F \|f\|_2 \leq 2(j+1)\|\bm\hat{U}_{j,1}\|_F^2 \|\bm\hat{s}\|_2\|(\bm\hat{U}_{j,1}^\ast \bm\hat{U}_{j,1})^{-1}\|_F\varepsilon. $$
Let $\sigma_{max}$ and $\sigma_{min}$ be the largest and the
smallest singular values of $\bm\hat{U}_{j,1}^\ast,$ respectively.
Then, we have $\|\bm\hat{U}_{j,1}\|_F^2 \leq (j+1)\sigma_{max}^2$ and $
\|(\bm\hat{U}_{j,1}^\ast \bm\hat{U}_{j,1})^{-1}\|_F \leq
\frac{j+1}{\sigma_{min}^2}.$ Substitute these inequalities in
the previous equation and use $\mathcal{K}_2(\bm\hat{U}_{j,1})^2=\frac{\sigma_{max}^2}{\sigma_{min}^2}$ to complete the proof.
\end{proof}

\ni In the next lemma, we use the Lemma-\ref{lem12} to bound the
norm of error vector in the second orthogonalization process in the
steps (k)~-~(n)  of the Algorithm-\ref{Improved TOAR method}.
Moreover, we are assuming that $\mathcal{K}_2(\bm\hat{U}_{j,1})$ is
moderately small.
\begin{lem}\label{lem13}\relax
Let $g_{1,j}$ be a overall floating point error vector resulting
from the second level orthogonalization process  in the steps
(k)-(n) of the Algorithm-\ref{Improved TOAR method}. Then
\begin{equation}\label{eq17b}\relax
\|g_{1,j} \|_2 \leq (j+1) \|\bm\hat{U}_{j+1,1}\|_F\|\underline{
\bm\hat{H}_k}(1:j+1,j)\|_2 \varepsilon+\mathcal{O}(\varepsilon)^2
\end{equation}
\end{lem}
\begin{proof}
On applying Lemma-\ref{lem4.1}(c) to the second orthogonalization
process in  the steps (k)~-~(n),  the computed $(j+1)$th column of
$U_{k+1,1}$ satisfies,
\begin{equation}\label{eq17a}\relax
\bm\hat{h}_{j+1,j} \bm \hat{U}_{j+1,1}(1:j,j+1)=\bm\hat{s}-
\bm\hat{U}_{j,1}(\bm\hat{U}_{j,1}^\ast
\bm\hat{U}_{j,1})^{-1}\bm\hat{U}_{j,1}^\ast\bm\hat{s}+g_{1,j}
=\bm\hat{s}- \bm\hat{U}_{j,1}\bm\hat{h}_j+g_{1,j}
\end{equation}
where $g_{1,j} $ is a floating point error vector
satisfying the inequality,
$$\|g_{1,j} \|_2 \leq (j+1) \|\bm\hat{U}_{j+1,1}\|_F\|\underline{ \bm\hat{H}_k}(1:j+1,j)\|_2 \varepsilon+ \varphi_2 \|\bm\hat{U}_{j,1}\|_F \|f\|_2 \varepsilon +\mathcal{O}(\varepsilon^2),$$
where $f$ is the error vector same as in the previous lemma.
In the above equation, we used the fact $h=
\gamma_1=(\bm\hat{U}_{j,1}^\ast
\bm\hat{U}_{j,1})^{-1}\bm\hat{U}_{j,1}^\ast\bm\hat{s}.$ 
Therefore, the proof is complete by using the fact from the
Lemma-\ref{lem12} that $\|\bm\hat{U}_{j,1}\|_F\|f\|_2 \varepsilon$ in the above
inequality is of $\mathcal{O}(\varepsilon^2).$
\end{proof}
\\\\\ni In the
following, we use the Lemma-\ref{lem13} to derive an upper bound for $\|F_1\|_F^2.$
\begin{lem}\label{lem14}\relax
Let $\bm\hat{Q}_{k},$ $\bm\hat{U}_{k}$ and $\underline{\hat{H}_k}$
be computed in the k-step of the I-TOAR procedure. Then
\begin{equation}\label{eq18b}\relax
\|F_1\|_F \leq
\varphi\|\bm\hat{Q}_{k}\|_2\|\bm\hat{U}_{k,1}\|_2\|\bm\hat{\underline{H_k}}\|_F\varepsilon+\mathcal{O}(\varepsilon^2)
\end{equation}
where $\varphi= (k+1)(2k+1).$
\end{lem}
\begin{proof}
Note that, in the $j$th iteration of I-TOAR, the computed quantity
at step-(d) is
$$\bm\hat{r}_j = fl  \Big([A~~B]\begin{bmatrix}\bm\hat{Q}_{j}\bm\hat{U}_{j,1}(:,j)\\ \hat{Q}_{j}\bm\hat{U}_{j,2}(:,j) \end{bmatrix}\Big)$$
Now, apply the Lemma-\ref{lem4.1}(c) to the orthogonalization and
normalization processes in the steps (e)~-~(h) and (x) of the
Algorithm-\ref{Improved TOAR method}. Then, the computed column
$\bm\hat{q}_{j+1}$ of $Q_k$ and $\bm\hat{s},$ $\bm\hat{\beta}$ computed in the steps (m) and (i) satisfy the following equation
\begin{equation}\label{eq18a}\relax
\bm\hat{\beta} \bm\hat{q}_{j+1}=\bm\hat{r}_j
-\bm\hat{Q}_{j}\bm\hat{s}- \tilde{f}_j,
\end{equation}
where 
  $\tilde{f}_j$ is the error vector, which satisfies the
following:
\begin{equation}\label{eq18lc}\relax
\|\tilde{f}_j\|_2 \leq (j+2) \| \bm\hat{Q}_{j+1}\|_F
\Big\|\begin{bmatrix} \bm\hat{s}\\\bm\hat{\beta}\end{bmatrix}
\Big\|\varepsilon +  \mathcal{O}(\varepsilon^2).
\end{equation}
Now, from the equation (\ref{eq17}), consider the
$j$th column of $F_1,$
\begin{equation}\label{eq19}\relax
f_{j,1} = \bm
\hat{r}_j-\bm\hat{Q}_k\bm\hat{U}_{k,1}(:,1:j)\bm\hat{h}_j-\bm\hat{h}_{j+1,j}\bm\hat{Q}_k
\bm\hat{U}_{k,1}(:,j+1)\\
=\bm\hat{r}_j
-\bm\hat{Q}_j\bm\hat{U}_{j,1}\bm\hat{h}_j-\bm\hat{h}_{j+1,j}\bm\hat{Q}_{j+1}\bm\hat{U}_{j+1,1}(:,j+1),
\end{equation}
where, for the second equality, we exploited the upper Hessenberg
structure of $\bm\hat{H}_k.$ Moreover, from the equation
(\ref{eq6}), we have
$$\bm\hat{h}_{j+1,j}\bm\hat{Q}_{j+1}\bm\hat{U}_{j+1,1}(:,j+1)= \bm\hat{h}_{j+1,j} \bm\hat{Q}_{j}\bm \hat{U}_{j+1,1}(1:j,j+1)+\bm\hat{q}_{j+1}\bm\hat{\beta}$$
Further, on left multiplying the equation (\ref{eq17a}) with
$\bm\hat{Q}_j,$ we have
\begin{equation}\label{eq19a}\relax
\bm\hat{h}_{j+1,j} \bm\hat{Q}_{j}\bm
\hat{U}_{j+1,1}(1:j,j+1)=\bm\hat{Q}_{j}\bm\hat{s}-\bm\hat{Q}_{j}\bm
\hat{U}_{j,1}\bm\hat{h}_{j}+\bm\hat{Q}_{j}g_{1,j}
\end{equation}
Hence, the substitution of the equations (\ref{eq18a}) and
(\ref{eq19a}) in the equation (\ref{eq19}) will give
$$f_{j,1}=\bm\hat{r}_j-\bm\hat{Q}_j\bm\hat{s}+\bm\hat{q}_{j+1}\bm\hat{\beta}+\bm\hat{Q}_{j}g_{1,j}\\
= \tilde{f}_j+\bm\hat{Q}_{j}g_{1,j}.$$ Thus,
\begin{eqnarray}\label{eq20}\relax
\nonumber \|f_{j,1}\|_2 \leq
\|\tilde{f}_j\|_2+\|\bm\hat{Q}_{j}\|_F\|g_{1,j}\|_2,
~~~~~~~~~~~~~~~~~~~~~~~~~~~~~~~~~~~~~~~~~~~~~~~~~~~~~~~~~~~~~~~~~~~~~~~~~\\
\nonumber
 \leq \Big((j+2)\Big\|\begin{bmatrix}\bm\hat{s}\\\bm\hat{\beta}
\end{bmatrix}\Big\|_2+\|g_{1,j}\|_2\big)\|\bm
\hat{Q}_{j+1}\|_F+\mathcal{O}(\varepsilon^2),~~~~~~~~~~~~~~~~~~~~~~~~~~~~~~~~~~~\\
\nonumber
\leq \Big((j+2)
\|\bm\hat{U}_{j+1,1}\|_2\|\bm\hat{\underline{ H_k}}(1:j+1,j)\|_2
\varepsilon+\|g_{1,j}\|_2\Big)\|\bm
\hat{Q}_{j+1}\|_F+\mathcal{O}(\varepsilon^2),~\\
\nonumber
 \leq
(2j+3)\|\bm\hat{Q}_{j+1}\|_F\|\bm\hat{U}_{j+1,1}\|_F\|\bm\hat{\underline{H_k}}(1:j+1,j)\|_2\varepsilon+\mathcal{O}(\varepsilon^2),~~~~~~~~~~~~~~~~~~~\\
\leq
(2k+1)\|\bm\hat{Q}_{k}\|_F\|\bm\hat{U}_{k,1}\|_F\|\bm\hat{\underline{H_k}}(1:j+1,j)\|_2\varepsilon+\mathcal{O}(\varepsilon^2).~~~~~~~~~~~~~~~~~~~~~~~~~~~
\end{eqnarray}
In the second, third and fourth inequalities, we used the
equations (\ref{eq18lc}), (\ref{eq17a}) and (\ref{eq17b}),
respectively, whereas the following were used to obtain the last inequation:
$$\|\bm\hat{Q}_{j+1}\|_F \leq \|\bm\hat{Q}_{k}\|_F, ~~ \|\bm\hat{U}_{j+1,1}\|_F\|\leq \|\bm\hat{U}_{k,1}\|_F\|, ~\mbox{and} ~~j+1 \leq k.$$
Now, from the equation (\ref{eq20}), $\|F_1\|_F$ is given by
$$\|F_1\|_F^2 =\displaystyle \sum \limits_{j=1}^{k-1}\|f_{j,1}\|_2^2 \leq (2k+1)^2 \displaystyle \sum \limits_{j=1}^{k-1}\|\bm\hat{Q}_{k}\|_F^2\|\bm\hat{U}_{k,1}\|_F^2\|\bm\hat{\underline{H_k}}(1:j+1,j)\|_2^2\varepsilon^2+\mathcal{O}(\varepsilon^3) $$
$$ = (2k+1)^2\|\bm\hat{Q}_{k}\|_F^2\|\bm\hat{U}_{k,1}\|_F^2\|\bm\hat{\underline{H_k}}\|_F^2\varepsilon^2+\mathcal{O}(\varepsilon^3)$$
Therefore, the proof will be complete by converting the Frobenius
norm to 2-norm.
\end{proof}

\ni \textit{Remark-1:}
When the complete reorthogonalization has applied, the above result holds true with little change in the coefficients. In this case, following a  similar procedure as in the remark-1 in \cite{stabtoar}, it is easy to see
that the norm of the error vector $\tilde{f}_j$ satisfies the
following relation:
$$\|\tilde{f}_j\|_2 \leq (2c(j+1)+1)\|\bm\hat{Q}_{j+1}\|_2 \|[\bm\hat{s}^T~\bm\hat{\beta}]\|_2\varepsilon+\mathcal{O}(\varepsilon^2),$$
where $c$ is a small constant.\\ 
We now present the main theorem for an upper bound of relative backward error in the I-TOAR procedure. Moreover, we are using the following assumption in the proof:
\begin{equation}\label{eq21}\relax
\Big\| \begin{bmatrix}|\bm\hat{Q}_{k-1}\bm\hat{U}_{k-1,1})^\dag
\\\bm\hat{Q}_{k-1}\bm\hat{U}_{k-1,2})^\dag  \end{bmatrix} \Big\|_2
\leq \zeta_1
\|(\bm\hat{Q}_{k-1}\bm\hat{U}_{k-1})^\dag\|_2~~\mbox{and}~~
\sigma_{min}({\bm\hat{U}_{k,1}}) =\zeta_2
\sigma_{min}({\bm\hat{U}_{k}}),
\end{equation}
where $\sigma_{min}(X)$ denotes the smallest singular value of
a matrix $X.$ 
\begin{thm}\label{thm1}\relax
Let $\bm\hat{Q}_k,~\bm\hat{U}_{k,1}$ and $\bm\hat{U}_{k,2}$ be matrices of 
full column rank. Let
$$\mathcal{K} =
max\{\mathcal{K}_2(\bm\hat{Q}_k),\mathcal{K}_2(\bm\hat{U}_{k})\},$$
where for any matrix $X,$ $\mathcal{K}_2(X)$ denotes its 2-norm condition number.
 If $(k+1)(2k+1)\mathcal{K}^4 (\zeta_1)/\zeta_2
\varepsilon <1,$ then
$$\frac{\|[\triangle A~~\triangle B]\|_F}{\|[A~~B]\|_F} \leq (2\zeta_1 nk^2+(\varphi_2/2)\mathcal{K}^2)\mathcal{K}^2 \varepsilon,$$
where $\varphi_2 =(k+1)(2k+1)\zeta_1/\zeta_2.$
\end{thm}
\begin{proof}
We have
$$\|[\triangle A~~\triangle B]\|_F = \|[\triangle A_1+\triangle A_2~~\triangle B_1+\triangle
B_2]\|_F \leq \|[\triangle A_1~~\triangle B_1]\|_F+\|[\triangle
A_2~~\triangle B_2]\|_F.$$ Using the definition of $\triangle A_1$
and $\triangle B_1$, we have the following for $\alpha=1/2.$
$$\|[\triangle A_1~~\triangle B_1]\|_F \leq 1/2 \|F_{mv}\|_F  \Big\|\begin{bmatrix}(\bm\hat{Q}_{k-1}\bm\hat{U}_{k-1,1})^\dag \\ (\bm\hat{Q}_{k-1}\bm\hat{U}_{k-1,2})^\dag\end{bmatrix}\Big\|_2$$
Similarly, using the definition of $\triangle A_2$ and $\triangle
B_2,$ we have
$$\|[\triangle A_2~~\triangle B_2]\|_F \leq 1/2 \|F_1\|_F  \Big\|\begin{bmatrix}(\bm\hat{Q}_{k-1}\bm\hat{U}_{k-1,1})^\dag \\ (\bm\hat{Q}_{k-1}\bm\hat{U}_{k-1,2})^\dag\end{bmatrix}\Big\|_2.$$
Further, by using the equation (\ref{eq21})
we have the following inequality:
$$\Big\|\begin{bmatrix}(\bm\hat{Q}_{k-1}\bm\hat{U}_{k-1,1})^\dag \\
(\bm\hat{Q}_{k-1}\bm\hat{U}_{k-1,2})^\dag\end{bmatrix}\Big\|_2 \leq
\zeta_1\|(\bm\hat{Q}_{k-1}\bm\hat{U}_{k-1})^\dag\|_2 \equiv
\frac{\zeta_1}{\sigma_{min}(\bm\hat{Q}_{k-1}\bm\hat{U}_{k-1})}$$
$$~~~~~~\leq \frac{\zeta_1}{\sigma_{min}(\bm\hat{Q}_{k-1})\sigma_{min}( \bm\hat{U}_{k-1})}  $$
In addition, with the bound for $\|F_{mv}\|_F$ in Lemma-\ref{lem11},
this gives the following upper bound for $\|[\triangle
A_1~~\triangle B_1]\|_F.$
\begin{eqnarray}\label{eq22}\relax
\nonumber |[\triangle A_1~~\triangle B_1]\|_F \leq
\frac{\zeta_12k^2n \|[A~~B]\|_F\|\bm\hat{Q}_{k-1}\|_2
\|\bm\hat{U}_{k-1}\|_2 }{
\sigma_{min}(\bm\hat{Q}_{k-1})\sigma_{min}(
\bm\hat{U}_{k-1})}\varepsilon + \mathcal{O}(\varepsilon^2)\\
 \leq \zeta_1 2k^2n
\|[A~~B]\|_F \mathcal{K}^2
\varepsilon+\mathcal{O}(\varepsilon^2)~~~~~~~~~~~~~~~~~~.
\end{eqnarray}
Now, recall from the equation (\ref{eq17c}) that
$$\bm\hat{\underline{H_k}} = (\bm\hat{Q}_{k}\bm\hat{U}_{k,1})^\dag
([A~~B]+[\triangle A~~\triangle
B])\begin{bmatrix}\bm\hat{Q}_{k-1}\bm\hat{U}_{k-1,1}\\\bm\hat{Q}_{k-1}\bm\hat{U}_{k-1,2}\end{bmatrix},$$
and repeatedly apply the inequality, $\|XY\|_F \leq \|X\|_2\|Y\|_F$
 to obtain the following:
$$\|\bm\hat{\underline{H_k}}\|_F \leq
\|(\bm\hat{Q}_{k}\bm\hat{U}_{k,1})^\dag\|_2\Big\|\begin{bmatrix}\bm\hat{Q}_{k-1}\bm\hat{U}_{k-1,1}\\\bm\hat{Q}_{k-1}\bm\hat{U}_{k-1,2}\end{bmatrix}\Big\|_2(\|[A~~B]\|_F+\|[\triangle
A~~\triangle B]\|_F).$$
Then, use the following inequalities in the
above equation,
$$\|(\bm\hat{Q}_{k}\bm\hat{U}_{k,1})^\dag\|_2 \leq \frac{1}{\sigma_{min}(\bm\hat{Q}_{k})\sigma_{min}(\bm\hat{U}_{k,1}) }~~\mbox{and}~~\Big\|\begin{bmatrix}\bm\hat{Q}_{k-1}\bm\hat{U}_{k-1,1}\\\bm\hat{Q}_{k-1}\bm\hat{U}_{k-1,2}\end{bmatrix}\Big\|_2 \leq
\|\bm\hat{Q}_{k-1}\|_2\|\bm\hat{U}_{k-1}\|_2,$$ to get the 
following:
$$\|\bm\hat{\underline{H_k}}\|_F \leq \frac{\|\bm\hat{Q}_{k-1}\|_2\|\bm\hat{U}_{k-1}\|_2}{\sigma_{min}(\bm\hat{Q}_{k})\sigma_{min}(\bm\hat{U}_{k,1}) }(\|[A~~B]\|_F+\|[\triangle
A~~\triangle B]\|_F).$$ 
By using this equation and the result in the
Lemma-\ref{lem14}, we have
\begin{eqnarray}\label{eq23}\relax
\nonumber\|F_1\|_F\Big\|\begin{bmatrix}(\bm\hat{Q}_{k-1}\bm\hat{U}_{k-1,1})^\dag \\
(\bm\hat{Q}_{k-1}\bm\hat{U}_{k-1,2})^\dag\end{bmatrix}\Big\|_2
~~~~~~~~~~~~~~~~~~~~~~~~~~~~~~~~~~~~~~~~~~~~~~~~~~~~~~~~~~~~~~~~~~~~~~~~~~~~~~~~~~~~~~~~~~~~~~~~~~~~~~~~~~~~
\\ \nonumber \leq
\varphi \frac{\zeta_1
\|\bm\hat{Q}_{k-1}\|_2\|\bm\hat{U}_{k-1}\|_2}{\sigma_{min}(\bm\hat{Q}_{k})\sigma_{min}(\bm\hat{U}_{k,1})
}\frac{\|\bm\hat{Q}_{k}\|_2\|\bm\hat{U}_{k}\|_2}{\sigma_{min}(\bm\hat{Q}_{k-1})\sigma_{min}(\bm\hat{U}_{k-1})
} (\|[A~~B]\|_F+\|[\triangle A~~\triangle B]\|_F)\varepsilon +
\mathcal{O}(\varepsilon^2)~~~~~~~~~~~~~~\\ \nonumber \leq \varphi
(\zeta_1/\zeta_2)
\mathcal{K}_2(\bm\hat{Q}_k)\mathcal{K}_2(\bm\hat{U}_{k})\mathcal{K}_2(\bm\hat{Q}_{k-1})\mathcal{K}_2(\bm\hat{U}_{k-1})(\|[A~~B]\|_F+\|[\triangle
A~~\triangle B]\|_F)\varepsilon +
\mathcal{O}(\varepsilon^2))~~~~~~~~~~~\\
\leq \varphi_2 \mathcal{K}^4 (\|[A~~B]\|_F+\|[\triangle A~~\triangle
B]\|_F)\varepsilon +\mathcal{O}\varepsilon^2)~~~~~~~~~~~~~~~~~~~~~~~~~~~~~~~~~~~~~~~~~~~~~~~~~~~~~~~~~~~~~~~~~~~~~~~~~~~~~
\end{eqnarray}
 where $\varphi_2=
(2k+1)(k+1)\zeta_1/\zeta_2.$ The two inequalities in the
equation (\ref{eq21}) were used for the first and second
inequalities, respectively . Now, by combining the 
equations (\ref{eq22}) and (\ref{eq23}), we have
$$\|[\triangle A~~\triangle B]\|_F \leq \frac{(2\zeta_1 n k^2+(\varphi_2/2) \mathcal{K}^2)\mathcal{K}^2 }{(1-(\varphi_2/2) \mathcal{K}^4  \varepsilon)}\|[A~~B]\|_F \varepsilon+\mathcal{O}(\varepsilon^2). $$
As we assumed $(\varphi_2/2)\mathcal{K}^4 \varepsilon < 1,$ the
theorem is proven by omitting the denominator, since it can be
covered by the term $\mathcal{O}(\varepsilon^2).$
\end{proof}

\ni\textit{Remark-2:} Similar to the Theorem-2.5 in \cite{stabtoar}, the previous lemmas were assumed that the matrices A and B are known explicitly so that the standard error bound for the matrix-vector multiplication applicable. The stability analysis of the Arnoldi method in \citep[Theorem-2.5]{31} also  used the same assumption. Otherwise, an error bound for matrix-vector multiplication depends on the specific formulation of  $A$ and $B.$ 

\section{Numerical examples}
In this section, we apply the I-TOAR procedure to the application of
Model order reduction of second-order dynamical systems. A continuous time invariant dynamical system is of the following form:
\begin{eqnarray}\label{N1}\relax
\nonumber M\ddot{x}(t)=-D\dot{x}(t)-Kx(t)+Fu(t)\\
y(t)= C_px(t)+C_v \dot{x}(t)
\end{eqnarray}
where $M,~D,~K \in \mathcal{R}^{n \times n}$ are mass, damping, and
stiffness matrices, respectively. $F \in \mathcal{R}^{n \times m},$
$C_p, ~C_v \in \mathcal{R}^{q \times n}$ are constant matrices. In
this paper, we are assuming $C_v=0$ and $m=q=1.$ Thus, we considering the following single input-single output dynamical
system of the form:
\begin{eqnarray}\label{N2}\relax
\nonumber M\ddot{x}(t)=-D\dot{x}(t)-K x(t)+fu(t)\\
y(t)= c x(t),
\end{eqnarray}
where $x(t)$ is the state vector, $u(t)$ is the input vector and
$y(t)$ is the output vector. Here, $f$ is a input distribution array
and $c$ is the outer measurement array. Moreover, for the convenience, we are
assuming that $x(0)= \dot{x}(0)=0.$  

Applying the Laplace transform
on both sides of the previous equation will give
\begin{eqnarray}\label{N4}\relax
\nonumber s^2 M X(s)+s DX(s)+K X(s)=fU(s)\\
Y(s)= c X(s).
\end{eqnarray}
where $X(s),Y(s),$ and $U(s)$ are Laplace transforms of $x(t),y(t),$
and $u(t),$ respectively. Thus, we have
$$Y(s) = c (s^2M+sD+K)^{-1}f U(s) .$$
$h(s):=c(s^2M+sD+K)^{-1}f$ is called as the \emph{Transfer function}. Using
the definition of $h(s),$ the previous equation can be written as
$$Y(s) = h(s)X(s).$$
The Model order reduction method produces a lower order dynamical
system that closely resembles the characteristics of the original system.
Though model order reduction is possible in many ways, in this paper, we are using the Galerkin projection based
reduction method. This method defines a projection operator 
using an orthonormal basis of a subspace generated by the mass,
damping and stiffness matrices. Then, it projects the system
(\ref{N4}) onto a subspace of smaller dimension. 
The following is the resulting reduced model of a system in the equation (\ref{N4}):
\begin{eqnarray}\label{new5}\relax
V^\ast (s^2 M V X_k(s)+s D V X_k(s)+K V X_k(s))=V^\ast f U(s)
\end{eqnarray}
The above system will be solved for $X_k(s).$ Note that $X_k(s)$ has a 
lesser number of elements compared to $X(s)$ in (\ref{N4}). Now, the
approximation to $Y(s)$ is given by $$Y_k(s):= h_k(s)X_k(s).$$ Here, 
$h_k(s)$ is an approximation to the Transfer function $h(s)$ and is
given by
$$h_k(s):= c_k (s^2M_k+sD_k+K_k)^{-1}V^\ast f,$$
where $c_k=cV,~M_k=V^\ast MV,~D_k=V^\ast DV $ and $K_k= V^\ast KV.$

The main objective of model order reduction techniques is to compute
$h_k(s),$ as an accurate approximation of $h(s)$ over a wide range
of frequency intervals around a prescribed shift $s_0.$ As in
\cite{stabtoar}, to meet this objective, we rewrite the transfer
function $h(s)$ by including the shift $s_0$ as follows:
$$h(s)= c\big((s-s_0)^2 M+(s-s_0)\tilde{D}+\tilde{K})^{-1}f$$
Using I-TOAR, we compute an orthonormal basis matrix $Q_k \in
\mathcal{R}^{n \times \eta_k}$ of the second order Krylov subspace
$${G}_k(-\tilde{K}^{-1}\tilde{D},-\tilde{K}^{-1}M;0,r_0=\tilde{K}^{-1}f)$$
where $\tilde{D} = 2s_0M+D$ and $\tilde{K}= s_0^2M+s_0D+K.$ Then, we
used  $Q_k$ in place of $V$ in the equation (\ref{new5}) for computing $h_k(s),$
an approximation to $h(s).$

Our numerical experiments compare the accuracy of reduced dynamical systems defined using the orthonormal basis matrices in the TOAR and I-TOAR methods. In both the TOAR and I-TOAR methods, we apply the reorthogonalization to ensure that the computed basis is orthonormal up to the machine precision. 
For I-TOAR, we used the complete reorthogonalization whereas, for the TOAR method 
we used the same setup as in \cite{stabtoar}. All algorithms are implemented using MATLAB and were run on a machine Intel(R) Core(TM)i7-4770
CPU@3.40GHz with 8GB RAM. For the convenience,
we use the same examples as in \cite{stabtoar}. The author
would acknowledge D.LU, an author of \cite{stabtoar} for
providing the data of these examples.
\begin{eg}
This example is a finite element model of a shaft on bearing supports with a
damper in MSC/NASTRAN. It is a second-order system and of dimension
400. The mass, damping matrices are symmetric, and the stiffness matrix
is symmetric positive definite. We use the expansion point $s_0=150\times 2\pi$ to approximate the Transfer function $h(s)$ over the frequency interval $[0,3000].$ 
\end{eg}
\vs{-0.75cm}
\begin{figure}[!htb]
\begin{center}
\includegraphics[width = 5.2in,height=2.7in]{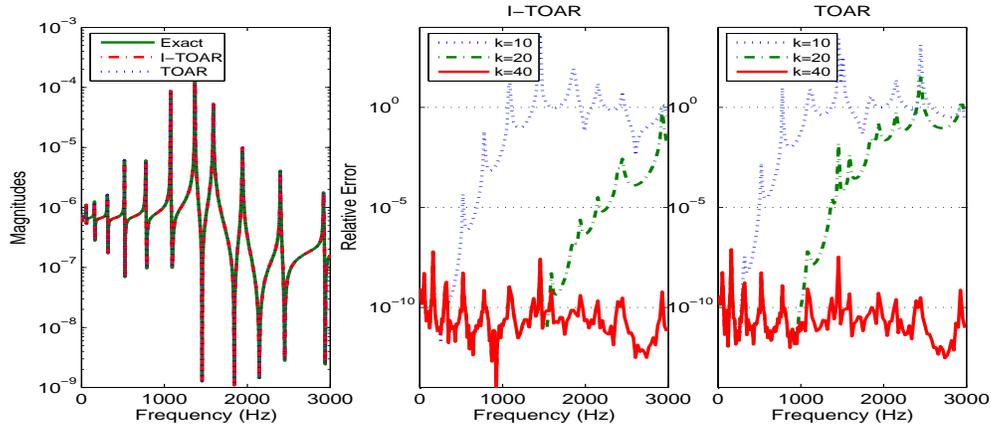}
\vs{-1cm}
\caption{Magnitudes of transfer functions h(s) and $h_k(s)$ with $k
= 40$ (left). Relative errors $|h(s) - h_k(s)|/|h(s)|$ for $k = 10,
20, 40 $(middle and right).}
\end{center}
\vs{-0.5cm}
\end{figure}

The left plot of Figure-1 shows the magnitudes of the transfer
function $h(s)$ of the full-order system, and the transfer functions
$h_k(s)$ of the reduced systems generated by the I-TOAR and TOAR
procedures for $k=40.$ The relative errors of the transfer functions
in the I-TOAR and TOAR procedures for $k=10,20$ and $40$ are shown in
the middle and right plots of the Figure-1  respectively. As we can
see that, the transfer functions $h_k(s)$ in the I-TOAR and TOAR methods
produces almost the same accuracy in the frequency interval
$[0,2000],$ for $k=10,20,$ and $40.$ In the frequency interval
$[2000,3000],$ the transfer function $h_k(s)$ by I-TOAR is a more
accurate approximation than the one produced by the TOAR method. Like the
TOAR method, in the I-TOAR method also, the approximation accuracy of
$h_k(s)$ is improved, when increasing $k$ from $10$ to $40.$
\begin{eg}
This example is the butterfly gyroscope problem from the Oberwolfach
collection. The full dynamical system is of the order $n=17361.$ The
mass and stiffness matrices $M$ and $K$ are symmetric. The damping matrix is of the form $D=\alpha M+\beta K.$ This second-order system
have 1 input vector and 12 output vectors. Following the experiments in \cite{stabtoar}, we considered the first output vector as the output
vector 'c.' The damping parameters $\alpha$ and $\beta$ are chosen
same as in \cite{stabtoar,li}, $\alpha = 0$ and $\beta = 10^{-7}.$
The expansion point $s_0$ also same as in \cite{stabtoar}, $s_0=1.05
\times 10^5.$
\end{eg}
\begin{figure}[!htb]
\begin{center}
\includegraphics[width = 5.2in,height=2.3in]{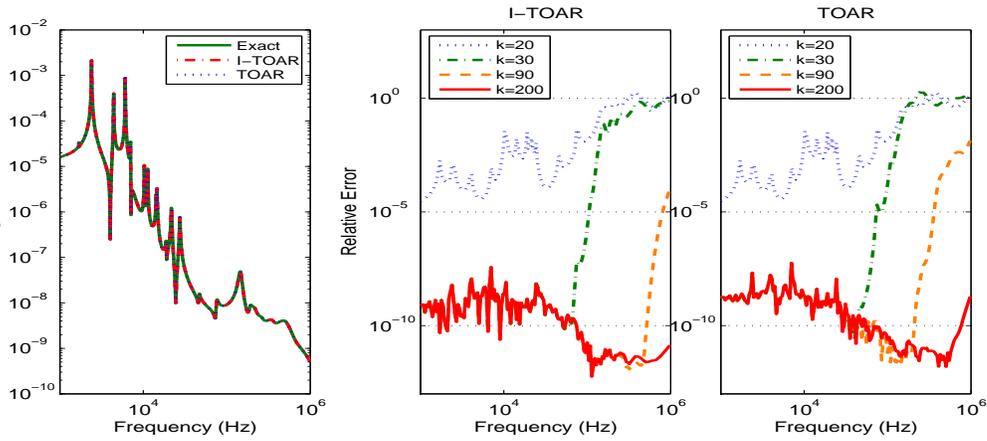}
\vs{0.5cm}\caption{Fig. 2. Magnitudes of transfer functions h(s) and
$h_k(s)$ with $k = 200$ (left). Relative errors $|h(s) -
h_k(s)|/|h(s)|$ for $k = 20, 30, 90,$ and $200$ (middle and right).}
\end{center}
\vs{-0.7cm}
\end{figure}

The magnitudes of the transfer functions shown in the left plot
of the Figure-2. The relative errors of transfer functions in I-TOAR and
TOAR show in the middle and right plots of the Figure-2. From
the figure, it is easy to observe the advantage of I-TOAR over
TOAR for the frequency range of $10^5-10^6 Hz$ for $k=90$ and $200.$
For $k=20$ and $30$, both I-TOAR and TOAR produced nearly the
same accuracy. We have observed the stagnation in both TOAR and
I-TOAR from $k=200$ onwards. From the middle and right plots of the
Figure-2, it is clear that the transfer function in I-TOAR is more
accurate than the transfer function in the TOAR procedure. Moreover, for
this example, we found that the quantity $\zeta_1$ of
Theorem-\ref{thm1} is of order $10^{11}$ and $\zeta_2=
1.000000000001526e+000.$ Further, found that condition number of the
computed matrices $\bm\hat{Q}_k$ and $\bm\hat{U}_k$ are equal to
$1+\mathcal{O}(\varepsilon),$ and these quantities satisfy the
condition $(\varphi_2 /2)\mathcal{K}^4 \varepsilon < 1.$

\section{Conclusion}
In this paper, we have proposed a new TOAR procedure.  It imposes an extra condition on the orthogonality of the matrices in the second-level orthogonalization of TOAR. Imposing the new condition gives
orthonormal basis of an associated linear Krylov subspace without any extra computation. A rigorous stability analysis has done on the proposed method. The backward analysis is in terms of the matrix $[A~~B]$ of the quadratic problem. It has shown that in the proposed method the second-order Krylov subspace of $[A~~B]$ embedded in that of $[A+\triangle A~~B+\triangle B]$ for sufficiently small $\| [\triangle A~~\triangle B]\|.$ This
problem was left open for TOAR in \cite{stabtoar}. Numerical experiments have shown that the basis matrices in I-TOAR are as accurate as ones in TOAR in the application of dimension reduction in second-order dynamical systems. The method proposed in this paper may help us to improve the methods for solving polynomial eigenvalue problems.


\section*{Acknowledgements}
The author
would like to acknowledge D.LU, author of \cite{stabtoar} for
providing the data used in the numerical examples. Further, this work was
supported by the National Board of Higher Mathematics, India under
Grant number 2/40(3)/2016/R\&D-II/9602.

\bibliographystyle{IMANUM-BIB}
\bibliography{ITOAR-IMANUMR1-refs}

\clearpage

\end{document}